\documentclass[a4paper,leqno,12pt]{amsart}
\usepackage[leqno]{amsmath}
\usepackage{amstext,amssymb,amsthm,enumitem}
\usepackage{graphicx}
\usepackage{tikz}
\usepackage[utf8]{inputenc}
\usepackage{float}

\usepackage{array}
\newcolumntype{C}[1]{>{\centering\let\newline\\\arraybackslash\hspace{0pt}}m{#1}}

\usepackage{pifont}
\newcommand{\cmark}{\ding{51}}%
\newcommand{\xmark}{\ding{55}}%

\newtheorem{theorem}{Theorem}

\newtheorem{lemma}{Lemma}
\newtheorem{proposition}{Proposition}
\newtheorem*{theorem*}{Theorem}
\theoremstyle{definition}
\newtheorem{example}{Example}
\newtheorem*{corollary*}{Corollary}
\newtheorem*{remark*}{Remark}

\tikzset{dots/.append style={ultra thick, fill=none}}

\begin{document}
	
	\title[Dichotomy property for maximal operators in non-doubling setting]{Dichotomy property for maximal operators in non-doubling setting}
	
	\author{Dariusz Kosz}
	\address{ 
		\newline Faculty of Pure and Applied Mathematics
		\newline Wroc\l{}aw University of Science and Technology 
		\newline Wyb. Wyspia\'nskiego 27 
		\newline 50-370 Wroc\l{}aw, Poland
		\newline \textit{Dariusz.Kosz@pwr.edu.pl}	
	}

	\begin{abstract} We investigate a dichotomy property for Hardy--Littlewood maximal operators, non-centered $M$ and centered $M^c$, that was noticed by Bennett, DeVore and Sharpley. 
	We illustrate the full spectrum of possible cases related to the occurrence or not of this property for $M$ and $M^c$ in the context of non-doubling metric measure spaces $(X, \rho, \mu)$. In addition, if $X = \mathbb{R}^d$, $d \geq 1$, and $\rho$ is the metric induced by an arbitrary norm on $\mathbb{R}^d$, then we give the exact characterization (in terms of $\mu$) of situations in which $M^c$ possesses the dichotomy property provided that $\mu$ satisfies some very mild assumptions. 
		
	\medskip	
	\noindent \textbf{2010 Mathematics Subject Classification.} Primary 42B25, 51F99
	
	\medskip
	\noindent \textbf{Key words:} maximal operator, dichotomy property, metric measure space, non-doubling measure.  	
	\end{abstract}
	
	\thanks{
	The author is supported by the National Science Centre of Poland, project no. 2016/21/N/ST1/01496.	
	} 
	
	\maketitle
	\section{Introduction}
	
	A dichotomy for the Hardy--Littlewood maximal operators was noticed for the first time by Bennett, DeVore and Sharpley in the context of the space of functions of bounded mean oscillation. In \cite{BDVS} the authors discovered the principle that for any function $f \in BMO(\mathbb{R}^d)$, $d \geq 1$, its maximal function $Mf$ either is finite almost everywhere or equals $+\infty$ on the whole $\mathbb{R}^d$. Later on, however, it turned out that this property is not directly related to the $BMO$ concept. Fiorenza and Krbec \cite{FK} proved that for any $f \in L^1_{loc}(\mathbb{R}^d)$ the following holds: if $Mf(x_0) < \infty$ for some $x_0 \in \mathbb{R}^n$, then $Mf(x)$ is finite almost everywhere. In turn, in \cite{AK} Aalto and Kinnunen have shown in a very elegant way that this implication remains true if one replaces the Euclidean space by any metric measure space with a doubling measure. Finally, some negative results in similar contexts also appeared in the literature. For example, in \cite{LSW} C.-C. Lin, Stempak and Y.-S. Wang observed that such a principle does not take place for local maximal operators. 
	
	The aim of this article is to shed more light on the above-mentioned issue by examining the occurrence of the dichotomy property for the two most common maximal operators of Hardy--Littlewood type, non-centered $M$ and centered $M^c$, associated with metric measure spaces for which the doubling condition fails to hold.
	
	By a metric measure space $\mathbb{X}$ we mean a triple $(X, \rho, \mu)$, where $X$ is a set, $\rho$ is a metric on $X$ and $\mu$ is a non-negative Borel measure. Throughout the paper we will additionally assume (without any further mention) that $\mu$ is such that $0 < \mu(B) < \infty$ holds for each open ball $B$ determined by $\rho$.
	
	In this context we introduce the $\textit{Hardy--Littlewood}$ $\textit{maximal operators}$, non-centered $M$ and centered $M^c$, by
	\begin{displaymath}
		Mf(x) = \sup_{B \ni x} \frac{1}{\mu(B)} \int_B |f| \, d \mu , \qquad x \in X,
	\end{displaymath}
	and
	\begin{displaymath}
	M^cf(x) = \sup_{r > 0} \frac{1}{\mu(B_r(x))} \int_{B_r(x)} |f| \, d\mu, \qquad x \in X,
	\end{displaymath}
	respectively. Here by $B$ we mean any open ball in $(X,\rho)$, while $B_r(x)$ stands for the open ball centered at $x \in X$ with radius $r>0$. We also require the function $f$ used above to belong to the space $L^1_{loc}(\mu)$ which means that $\int_B |f| \, d \mu < \infty$ for any ball $B \subset X$.
	
	We say that $M$ possesses the $\textit{dichotomy property}$ if for any $f \in L^1_{loc}(\mu)$ exactly one of the following cases holds: either $\mu(E_\infty(f)) = 0$ or $E_\infty(f) = X$, where $E_\infty(f) = \{x \in X \colon Mf(x) = \infty\}$. Similarly, $M^c$ possesses the dichotomy property if for any $f \in L^1_{loc}(\mu)$ we have either $\mu(E^c_\infty(f)) = 0$ or $E^c_\infty(f) = X$, where $E^c_\infty(f) = \{x \in X \colon M^cf(x) = \infty\}$. Notice that, equivalently, the dichotomy property can be formulated in the following way: if $Mf(x_0) < \infty$ (respectively, $M^cf(x_0) < \infty$) for some $f \in L^1_{loc}(\mu)$ and $x_0 \in X$, then $Mf$ (respectively, $M^cf$) is finite $\mu$-almost everywhere. 
	
	Observe that for any $f \in L^1_{loc}(\mu)$ we have $E^c_\infty(f) \subset E_\infty(f)$. Moreover, if the space is doubling (which means that $\mu(B_{2r}(x)) \lesssim \mu(B_r(x))$ holds uniformly in $x \in X$ and $r > 0$), then $E^c_\infty(f) = E_\infty(f)$. Nevertheless, at first glance, there is no clear reason why the two properties mentioned in the previous paragraph would be somehow interdependent in general, since $Mf$ and $M^cf$ may be incomparable if $(X, \rho, \mu)$ is not doubling. In other words, we have no obvious indications at this point that the existence or absence of the dichotomy property for one operator implies its existence or absence for another one. Therefore, natural problems arise: ``can each of the four possibilities actually take place for some metric measure space?" and ``can we additionally demand that this space be non-doubling?". Thus, one of the two major results in this article is to prove the following theorem that gives affirmative answers to these two questions.
	
	\begin{theorem}
		For each of the four possibilities regarding whether $M$ and $M^c$ possess the dichotomy property or not, there exists a non-doubling metric measure space for which the associated maximal operators behave just the way we demand.
	\end{theorem}
	
	\begin{proof}
		Examples 1, 2, 3 and 4 in Sections 2 and 3 together constitute the proof of this theorem, illustrating all the desired situations.
	\end{proof}
	
It is worth noting at this point that, in addition to indicating appropriate examples, our goal is also to ensure that they are constructed as simply as possible. Thus, in all examples presented later on $X$ is either $\mathbb{R}^d$ or $\mathbb{Z}^d$, $d \geq 1$, while $\rho$ is the standard Euclidean metric $d_e$ or the supremum metric $d_\infty$. Finally, in the discrete setting $\mu$ is defined by letting the value $\mu(\{x\}) > 0$ to each point $x \in X$, while in the continuous situation $\mu$ is determined by a suitable strictly positive weight $w$.

For the convenience of the reader, the results obtained in Examples 1, 2, 3 and 4 have been summarized in Table 1 below.

\begin {table}[H]
\vspace*{0.3cm}
\caption {Occurrence of the dichotomy property (DP) for $M$ and $M^c$ associated with spaces described in Examples 1, 2, 3 and 4.} \label{T}

\begin{center}
	\begin{tabular}{ | c | C{0.7cm} | C{0.7cm} | c | c | c |}
		
		\hline
		 & $X$ & $\rho$ & $\mu$ & DP for $M$ & DP for $M^c$ \\ \hline
		
		Ex. $1$ & $\mathbb{R}$ & $d_e$ & $e^{x^2} dx$ & \cmark &  \xmark \\ \hline
		
		Ex. $2$ & $\mathbb{R}$ & $d_e$ & $e^{-x^2} dx$ & \cmark & \cmark \\ \hline
		
		Ex. $3$ & $\mathbb{Z}^2$ & $d_\infty$ & 
		
		$\mu(n,m) = \left\{ \begin{array}{rl}
		4^{|m|} & \textrm{if } n = 0,  \\
		1 & \textrm{otherwise. }  \end{array} \right.$ & 
		
		\xmark &  \cmark \\ \hline
		
		Ex. $4$ & $\mathbb{Z}^2$ & $d_\infty$ & 
		
		$\mu(n,m) = \left\{ \begin{array}{rl}
		4^{|m|} & \textrm{if } n = 0,  \\
		2^{n^2} & \textrm{if } n < 0 \textrm{ and } m = 0,  \\
		1 & \textrm{otherwise. }  \end{array} \right. $ & 
		
		\xmark & \xmark \\ \hline
	\end{tabular}
\end{center}
\end{table}

One more comment is in order here. While the doubling condition for measures is often assumed in the literature to provide that most of the classical theory works, some statements can be verified under the less strict condition that the space is geometrically doubling or satisfies both geometric doubling and upper doubling properties (see \cite{H} for the details). In our case, although the metric measure spaces appearing in Table 1 are non-doubling, the corresponding metric spaces are geometrically doubling. This means that the general result for the class of doubling spaces, concerning the existence of the dichotomy property for maximal operators, cannot be repeated in the context of geometrically doubling spaces. Finally, Example 5 in Section 4 illustrates the situation where the space is geometrically doubling and upper doubling at the same time, while the associated operator $M$ does not possess the dichotomy property.   
	 
	\section{Real line case}
In this section we study the dichotomy property for the Hardy--Littlewood maximal operators $M$ and $M^c$ associated with the space $(\mathbb{R}, d_e, \mu)$, where $\mu$ is arbitrary. Let us note here that we consider one-dimensional spaces separately, since they have some specific properties, mainly due to their linear order (for example, in this case $M$ always satisfies the weak type $(1,1)$ inequality with constant $2$). Our first task is to prove the following. 

\begin{proposition}
	Consider the space $(\mathbb{R}, d_e, \mu)$, where $\mu$ is an arbitrary Borel measure. Then $M$ possesses the dichotomy property.
\end{proposition}
\noindent The proof of Proposition 1 is preceded by some additional considerations.  

Let $r(B)$ be the radius of a given ball $B$. For $f \in L^1_{loc}(\mu)$ we denote
\begin{displaymath}
L_f = L_f(\mu) = \Big\{ x \in \mathbb{R} \colon \lim_{r \rightarrow 0} \sup_{B \ni x \colon r(B)=r} \frac{1}{\mu(B)} \int_{B} | f(y) - f(x) | \, d\mu(y) = 0 \Big\}, 
\end{displaymath}
and
\begin{displaymath}
L^c_f = L^c_f(\mu) = \Big\{ x \in \mathbb{R} \colon \lim_{r \rightarrow 0} \frac{1}{\mu(B_{r}(x))} \int_{B_{r}(x)} | f(y) - f(x) | \, d\mu(y) = 0 \Big\}. 
\end{displaymath}
Notice that there is a small nuisance here, because $f$ is actually an equivalence class of functions, while $L_f $ and $L^c_f$ clearly depend on the choice of its representative. Nevertheless, for any two representatives $f_1$ and $f_2$ of a fixed equivalence class we have $\mu(L_{f_1} \triangle L_{f_2})=0$ and $\mu(L^c_{f_1} \triangle L^c_{f_2})=0$ (where $\triangle$ denotes the symmetric difference of two sets) and this circumstance is sufficient for our purposes. 

The conclusion of the following lemma is a simple modification of the well known fact about the set of Lebesgue points of a given function. Although the proof is rather standard, we present it for completeness (cf. \cite[Theorem 3.20]{Fo}).  
\begin{lemma}
Consider the space $(\mathbb{R}, d_e, \mu)$ and let	$f \in L^1_{loc}(\mu)$. Then $\mu(\mathbb{R} \setminus L_f) = 0$. 
\end{lemma}

\begin{proof}
	For a function $g \in L^1_{loc}(\mu)$ let us introduce the sets $L_{g,N}$, $N \in \mathbb{N}$, defined by
	\begin{displaymath}
	L_{g,N} = \Big\{ x \in \mathbb{R} \colon \limsup_{r \rightarrow 0} \sup_{B \ni x \colon r(B)=r} \frac{1}{\mu(B)} \int_{B} | g(y) - g(x) | \, d\mu(y) \leq \frac{1}{N} \Big\}.
	\end{displaymath}
	Note that $L_f = \bigcap_{N = 1}^{\infty} L_{f,N}$. Therefore, it suffices to prove that for each $N \in \mathbb{N}$ there exists a Borel set $A_N$ such that $(-N, N) \setminus L_{f,N} \subset A_N$ and $\mu(A_N) \leq 1/N$. 
	
	Fix $N$ and consider $f_N = f \cdot \chi_{(-N-1, N+1)}$. Thus $f_N \in L^1(\mu)$ and $L_{f_N,N}$ coincides with $L_{f,N}$ on $(-N, N)$. We take a continuous function $g_N$ satisfying $\|f_N - g_N \|_{L^1(\mu)} \leq 1 / (9N^2)$ (notice that continuous functions are dense in $L^1(\mu)$ by \cite[Proposition 7.9]{Fo}) and define two auxiliary sets
	\begin{displaymath}
	E_N^1 = \{x \in \mathbb{R} \colon |(f_N - g_N)(x)| > \frac{1}{3N} \}, \qquad E_N^2 = \{x \in \mathbb{R} \colon M(f_N - g_N)(x) > \frac{1}{3N} \}.
	\end{displaymath} 
	Observe that $\mu(E^1_N) \leq 1 / (3N)$ and $\mu(E^2_N) \leq 2 / (3N)$. Now we fix $x_0 \in (-N, N) \setminus (E_N^1 \cup E_N^2)$ and take $0 < \epsilon <1$ such that $|g_N(y) - g_N(x_0)| \leq 1 / (3N)$ for $|y - x_0| < \epsilon$. If $B$ contains $x_0$ and satisfies $r(B)< \epsilon/2$, then by using the estimate
	\begin{displaymath}
	| f(y) - f(x_0) | \leq | f_N(y) - g_N(y) | + |g_N(y) - g_N(x_0)| + |(g_N(x_0) - f_N(x_0)|,
	\end{displaymath}
	which is valid for all $y \in B$, we obtain 
	\begin{displaymath}
	\frac{1}{\mu(B)} \int_{B} | f(y) - f(x_0) | \, d\mu(y)
	\leq M(f_N - g_N)(x_0) + \frac{1}{3N} + |f_N(x_0) - g_N(x_0)| \leq \frac{1}{N},
	\end{displaymath}
	and therefore $A_N = E_N^1 \cup E_N^2$ satisfies the desired conditions. 
\end{proof}

\begin{remark*}
	Of course, the definitions of $L_f$ and $L_f^c$ can also be adapted to the situation of an arbitrary metric measure space $(X, \rho, \mu)$. In this case we have $\mu(X \setminus L_f) = 0$ (respectively, $\mu(X \setminus L_f^c) = 0$) for a given function $f \in L^1_{loc}(\mu)$ if only the associated maximal operator $M$ (respectively, $M^c$) is of weak type $(1,1)$ and continuous functions are dense in $L^1(\mu)$. This is the case, for example, when dealing with $L_f^c$ and the space $(\mathbb{R}^d, \rho, \mu)$, $d \geq 1$, where $\rho$ is the metric induced by a fixed norm (in particular, $\rho = d_e$ and $\rho = d_\infty$ are included) and $\mu$ is arbitrary. We explain some details more precisely in Section 4.
\end{remark*}

Now we are ready to prove Proposition 1.

\begin{proof}
	Assume that $\mu(E_\infty(f)) > 0$. Then we can take $x \in L_f$ such that $Mf(x) = \infty$. There exist balls $B_n$, $n \in \mathbb{N}$, containing $x$ and satisfying
	\begin{displaymath}
	\frac{1}{\mu(B_n)} \int_{B_n} |f(y)| \, d\mu(y) > n.
	\end{displaymath}
	Fix $\epsilon > 0$ such that
	\begin{displaymath}
	\frac{1}{\mu(B)} \int_{B} | f(y) - f(x) | \, d\mu(y) < 1,
	\end{displaymath}
	if $r(B) \leq \epsilon$ and denote $\delta = \min\{ \mu((x-\epsilon/2, x]), \mu([x, x + \epsilon/2))\}$. We obtain that $B_n \subsetneq (x-\epsilon/2, x + \epsilon/2)$ if $n \geq |f(x)| + 1$ and, as a result, $\mu(B_n) \geq \delta$ for that $n$.
	
	Now let us fix an arbitrary point $x' > x$ (the case $x < x'$ can be considered analogously). We denote $\gamma = \mu((x, x'+1)) < \infty$ and $B_n' = B_n \cup (x, x'+1)$, $n \in \mathbb{N}$. Observe that if $n \geq |f(x)| + 1$, then the set $B_n'$ forms a ball containing $x'$ and therefore
	\begin{displaymath}
	Mf(x') \geq \frac{1}{\mu(B_n')} \int_{B_n'} |f(y)| \, d\mu(y) \geq \frac{\mu(B_n)}{\mu(B_n')} \, \frac{1}{\mu(B_n)} \int_{B_n} |f(y)| \, d\mu(y) \geq \frac{\delta n}{\delta + \gamma}.
	\end{displaymath}
	This, in turn, implies $Mf(x') = \infty$, since $n$ can be arbitrarily large. 
\end{proof} 
	
	At the end of this section we show an example of a space $(\mathbb{R}, d_e, w(x) dx)$, where $w$ is a suitable weight (and $w(x) dx$ is non-doubling), for which the centered Hardy--Littlewood maximal operator does not possess the dichotomy property.
	
	\begin{example}
		Consider the space $(\mathbb{R}, d_e, \mu)$ with $d \mu = e^{x^2} dx$. Then $M$ possesses the dichotomy property, while $M^c$ does not.   
	\end{example}
	
	Indeed, it suffices to prove only the second part, since $M$ possesses the dichotomy property by Proposition 1. Consider $f(x) = x \cdot \chi_{(0, \infty)} (x)$. We shall show that $M^c(f) = \infty$ if and only if $x \geq 0$. 
	
	For $x \in \mathbb{R}$ and $r>0$ let us introduce the quantity
	\begin{displaymath}
	A_rf(x) = \frac{1}{\mu(B_r(x))} \int_{B_r(x)} |f(y)| \, e^{y^2} \, dy.
	\end{displaymath}
	At first, observe that $\lim_{r \rightarrow \infty} A_rf(0) = \infty$. Indeed, fix $N \in \mathbb{N}$ and take $r_0 > N$ such that
	\begin{displaymath}
	\int_{(N, r)} e^{x^2} \, dx \geq \frac{1}{3} 	\int_{(-r, r)} e^{x^2} \, dx, 
	\end{displaymath}
	 for each $r \geq r_0$. Therefore, for that $r$, we obtain
	 \begin{displaymath}
	 A_rf(0) = \frac{1}{\mu(B_r(0))} \int_{B_r(0)} f(x) \, e^{x^2} \, dx \geq \frac{N}{\mu(B_r(0))} \int_{(N, r)} e^{x^2} \, dx \geq \frac{N}{3},
	 \end{displaymath}
	 and thus $M^cf(0) = \infty$. Next, it is easy see that for any $x > 0$ there is $A_rf(x) \geq A_{r+x}f(0)$ for $r \geq x$. This fact, in turn, gives $Mf^c(x) = \infty$ for any $x \geq 0$.
	 
	 Now we show that $M^cf(x) < \infty$ if $x$ is strictly negative. Fix $x < 0$ and $r > 0$. We can assume that $r > |x|$, since for the smaller values of $r$ we have $A_rf(x) = 0$. Observe that it is possible to choose $r_0 > |x|$ such that for each $r \geq r_0$ 
	 \begin{displaymath}
	 e^{(x+r)^2} \leq 2 \, |x| \, e^{r^2}.
	 \end{displaymath}
	 If $r < r_0$, then $A_rf(x) \leq f(x + r_0)$. On the other hand, if $r \geq r_0$, then
	 \begin{displaymath}
	 A_rf(x) \leq \frac{1}{\mu(B_r(x))} \int_{B_r(x)} f(x) \, e^{x^2} \, dx \leq \frac{e^{(x+r)^2}}{2 \, \mu((x-r, -r))} \leq \frac{e^{(x+r)^2}}{2 \, |x| \, e^{r^2}} \leq 1,
	 \end{displaymath}
	 which implies $M^cf(x) < \infty$.       

	\section{Multidimensional case}
	
	Throughout this section we work with spaces that do not necessarily have a linear structure. In the first place, we would like to receive that in certain circumstances $M^c$ must possess the dichotomy property. Of course, for our purpose, we should ensure that the introduced criterion is relatively easy to apply and returns positive results also for some non-doubling spaces. Fortunately, it turns out that it is possible to find a condition that successfully meets all these requirements. 
	
	The following proposition is embedded in the context of Euclidean spaces, but it is worth keeping in mind that, in fact, it concerns all spaces $(X, \rho, \mu)$ for which $\mu(X \setminus L_f^c) = 0$ holds for each $f \in L^1_{loc}(\mu)$. 
	
	\begin{proposition}
	Consider the space $(\mathbb{R}^d, d_e, \mu)$, $d \geq 1$, and assume that
	\begin{equation}\label{C}
	\exists y_0 \in \mathbb{R}^d \colon \limsup_{r \rightarrow \infty} \frac{\mu(B_{r+1}(y_0))}{\mu(B_{r}(y_0))} = \tilde{C} = \tilde{C}(y_0) < \infty.
	\end{equation}
	Then the associated maximal operator $M^c$ possesses the dichotomy property. 	
	\end{proposition}
	
	Observe that condition (\ref{C}) is related to certain global properties of a given metric measure space $\mathbb{X}$ and thus its occurrence (or not) should be independent of the choice of the point $y_0$ specified above. Indeed, it can be easily shown that if the inequality in (\ref{C}) holds for some $y_0$, then it is also true if we replace $y_0$ by an arbitrary point $y \in X$. 
	
	Secondly, as it turns out according to Theorem 2 in Section 4, the converse also holds in the case $\mathbb{X} = (\mathbb{R}^d, d_e, \mu)$. Namely, we shall prove that if $M^c$ possesses the dichotomy property, then (\ref{C}) holds for some $y_0 \in \mathbb{R}^d$. Notice that we state only one of the implications in Proposition 2 above because it is enough to prove Theorem 1. On the other hand, the opposite implication allows us to say that the formulated condition is sufficient and necessary at the same time and, since looking for such conditions is interesting itself, we discuss it in a separate section.
	
	\begin{proof}
		Let $f \in L^1_{loc}(\mu)$ and assume that $\mu(E_\infty^c(f)) > 0$. We take $x_0 \in L_f^c$ such that $M^cf(x_0) = \infty$. Hence for each $n \in \mathbb{N}$ we have a ball $B_n = B_{r_n}(x_0)$ satisfying 
		\begin{displaymath}
		\frac{1}{\mu(B_n)} \int_{B_n} |f(y)| \, d\mu(y) > n.
		\end{displaymath}
		Fix $\epsilon >0$ such that
		\begin{displaymath}
		\frac{1}{\mu(B_r(x_0))} \int_{B_r(x_0)} |f(y)-f(x_0)| \, d\mu(y) \leq 1,
		\end{displaymath} 
		for $r \leq \epsilon$ and denote $\delta = \mu(B_\epsilon(x_0))$. If $n \geq |f(x_0)| + 1$, then $B_n \subsetneq B_\epsilon(x_0)$ and, as a result, we have $\mu(B_n) \geq \delta$.	This fact easily implies that $\lim_{n \rightarrow \infty} r_n = \infty$, since $f$ is locally integrable.
		
		Now we fix any point $x \in \mathbb{R}^d$. There exists $r_0>0$ such that
		\begin{displaymath}
		\mu(B_{r+1}(y_0)) \leq 2 \tilde{C} \, \mu(B_{r}(y_0)), 
		\end{displaymath}
		for each $r \geq r_0$. We choose $n_0 \geq |f(x_0)| + 1$ large enough to ensure that $n \geq n_0$ implies $r_n - |y_0 - x_0| \geq r_0$. Consider the balls $B_n' = B_{r_n + |x_0-x|}(x)$ for $n \in \mathbb{N}$. If $n \geq n_0$, then
		\begin{displaymath}
		\mu(B_n') \leq \mu(B_{r_n + |x_0-x| + |y_0 - x|}(y_0)) \leq (2\tilde{C})^m \mu(B_{r_n - |x_0 - y_0|}(x_0)) \leq (2\tilde{C})^m \mu(B_n),
		\end{displaymath}
		where $m > |x_0-x| + |y_0 - x| + |x_0 - y_0|$ is a positive integer independent of $n$. Finally, by using the fact that $B_n \subset B_n'$, we get
		\begin{displaymath}
		M^cf(x) \geq \frac{1}{\mu(B_n')} \int_{B_n'} |f(y)| \, d\mu(y) \geq \frac{\mu(B_n)}{\mu(B_n')} \frac{1}{\mu(B_n)} \int_{B_n} |f(y)| \, d\mu(y) \geq \frac{n}{(2\tilde{C})^m},
		\end{displaymath}
		which gives $M^cf(x) = \infty$, since $n$ can be arbitrarily large. 
	\end{proof}
	
	\begin{remark*}
	Notice that the conclusion of Proposition 2 remains true if we take the metric $d_\infty$ instead of $d_e$ provided that this time the balls determined by $d_\infty$ are used in (\ref{C}). There are also no obstacles to getting a discrete counterparts of the above statements. Namely, one can replace $\mathbb{R}^d$ by $\mathbb{Z}^d$, $d \geq 1$, and obtain the desired result for the space $(\mathbb{Z}^d, \rho, \mu)$, where $\rho = d_e$ or $\rho = d_\infty$ and $\mu$ is arbitrary. 
	\end{remark*}
	
	Now, with Propositions 1 and 2 in hand, we can easily give an example of a non-doubling space, for which both $M$ and $M^c$ possess the dichotomy property. 
	
	\begin{example}
		Consider the space $(\mathbb{R}, d_e, \mu)$ with $d\mu(x) = e^{-x^2} dx$. Then both $M$ and $M^c$ possess the dichotomy property.
	\end{example}
	
	Indeed, $M$ possesses the dichotomy property by Proposition 1, while $M^c$ possesses the dichotomy property by Proposition 2, since $\lim_{r \rightarrow \infty} \mu(B_{r+1}(0)) / \mu(B_r(0)) = 1$. \newline
	
	At this point, a natural question arises: will we get the same result for Gaussian measures in higher dimensions? The following proposition settles affirmatively this problem.  
	
	\begin{proposition}
		Consider the space $(\mathbb{R}^d, d_e, \mu)$ with $\mu(\mathbb{R}^d) < \infty$. Assume that $\mu$ is determined by a strictly positive weight $w$ satisfying
	\begin{equation}\label{D}
	0 < c_n \leq w(x) \leq C_n < \infty, \qquad x \in B_n(0), \, n \in \mathbb{N},
	\end{equation}
	for some numerical constants $c_n$ and $C_n$, $ n \in \mathbb{N}$. Then the associated maximal operators, $M$ and $M^c$, both possess the dichotomy property.
	\end{proposition}
	
	\begin{proof}
		It suffices to prove that $M$ possesses the dichotomy property, since $\mu(\mathbb{R}^d) < \infty$ implies that \eqref{C} is satisfied with $\tilde{C} = 1$ (regardless of which point $y_0 \in \mathbb{R}^d$ we choose).
		
		Take $f \in L_{loc}^1(\mu)$. We shall show that $\mu(\mathbb{R}^d \setminus L_f) = 0$. For a fixed $n \in \mathbb{N}$ let us consider the measure $\mu_n$ determined by $w_n$ satisfying
		\begin{displaymath}
		w_n(x) = \left\{ \begin{array}{rl}
		
		w(x) & \textrm{if } x \in B_n(0),  \\
		
		1 & \textrm{otherwise. }  \end{array} \right. 
		\end{displaymath}
	Observe that condition \eqref{D} implies that $\mu_n$ is doubling. Let $f_n = f \chi_{B_n(0)}$. We have
	\begin{displaymath}
	\mu(B_n(0) \setminus L_f) = \mu_n(B_n(0) \setminus L_{f_n}(\mu_n)) \leq \mu_n(\mathbb{R}^d \setminus L_{f_n}(\mu_n))= 0,
	\end{displaymath}
	because $f_n \in L^1_{loc}(\mu_n)$ and this yields $\mu(\mathbb{R}^d \setminus L_f) = 0$, since $n$ can be arbitrarily large.
	
	Assume that $\mu(E_\infty(f)) > 0$ and take $x_0 \in L_f$ such that $Mf(x_0) = \infty$. For each $n \in \mathbb{N}$ we have a ball $B_n \ni x_0$ for which
	\begin{displaymath}
	\frac{1}{\mu(B_n)} \int_{B_n} |f(y)| \, d\mu(y) > n.
	\end{displaymath}
	Fix $\epsilon >0$ such that
	\begin{displaymath}
	\frac{1}{\mu(B)} \int_{B} |f(y)-f(x_0)| \, d\mu(y) \leq 1,
	\end{displaymath} 
	whenever $B \subset B_\epsilon(x_0)$. If $n \geq |f(x_0)| + 1$, then $B_n \subsetneq B_\epsilon(x_0)$. Thus, combining condition \eqref{D} with the fact that $r(B_n) \geq \epsilon / 2$ for that $n$, we conclude that $\mu(B_n) \geq \delta$, where $\delta = \delta(x_0, \epsilon)$ is strictly positive and independent of $n$.
	
	Now we fix any point $x \in \mathbb{R}^d$ and take $n \geq |f(x_0)| + 1$. Let $B_n'$ be any ball containing $x$ and $B_n$. Then we get
	\begin{displaymath}
	Mf(x) \geq \frac{1}{\mu(B_n')} \int_{B_n'} |f(y)| \, d\mu(y) \geq \frac{1}{\mu(\mathbb{R}^d)} \int_{B_n} |f(y)| \, d\mu(y) \geq \frac{\delta n}{\mu(\mathbb{R}^d)},
	\end{displaymath}
	which gives $M^cf(x) = \infty$, since $n$ can be arbitrarily large. 
	\end{proof}
	
		
	Until now we furnished examples illustrating two of the four possibilities related to the problem of possessing or not the dichotomy property by $M$ and $M^c$. Notice that in both considered situations the indicated space was $\mathbb{R}$ with the usual metric and measure determined by a suitable weight. Unfortunately, as was indicated in Proposition 1, such examples cannot be used to cover the remaining two cases, since this time we want $M$ to not possess the dichotomy property. Therefore, a natural step is to try to use $\mathbb{R}^2$ instead of $\mathbb{R}$. This idea turns out to be right. However, for simplicity, the other two examples will be initially constructed in the discrete setting $\mathbb{Z}^2$. Also, for purely technical reasons, the metric $d_e$ is replaced by $d_\infty$. Nevertheless, after presenting Examples 3 and 4, we include some additional comments in order to convince the reader that it is also possible to obtain the desired results for the appropriate metric measure spaces of the form $(\mathbb{R}^2, d_e, \mu)$.
	
	While dealing with $\mathbb{Z}^2$, for the sake of clarity, we will write $B_r(n,m)$ and $\mu(n,m)$ instead of $B_r((n,m))$ and $\mu(\{(n,m)\})$, respectively. 
	
	
	\begin{example}
		Consider the space $(\mathbb{Z}^2, d_\infty, \mu)$, where $\mu$ is defined by 
	\begin{displaymath}
	\mu(n,m) = \left\{ \begin{array}{rl}

	4^{|m|} & \textrm{if } n = 0,  \\
	
	1 & \textrm{otherwise. }  \end{array} \right. 
	\end{displaymath}
	Then $M^c$ possesses the dichotomy property, while $M$ does not.
	\end{example}
	
	At first, observe that $M^c$ possesses the dichotomy property by Proposition 2 (or, more precisely, by the remark following Proposition 2), since 
	\begin{displaymath}
	\lim_{r \rightarrow \infty} \frac{\mu(B_{r+1}(0,0))}{\mu(B_{r}(0,0))} = 4.
	\end{displaymath}
	To verify the second part of the conclusion let us consider the function $f$ defined by
	\begin{displaymath}
	f(n,m) = \left\{ \begin{array}{rl}
	2^n & \textrm{if } n > 0 \textrm{ and } m=0,  \\
	
	0 & \textrm{otherwise. }  \end{array} \right. 
	\end{displaymath}
	We will show that $Mf(1,0) = \infty$ and $Mf(-1,0) < \infty$ (in fact, it should be clear for the reader that $(1,0)$ and $(-1, 0)$ may be replaced by any other points $(n_1,m_1)$ and $(n_2,m_2)$ such that $n_1$ is strictly positive and $n_2$ is strictly negative).
	
	Consider the balls $B_N = B_N(N,0)$ for $N \in \mathbb{N}$. Observe that
	\begin{displaymath}
	Mf(1,0) \geq \frac{1}{\mu(B_N)} \sum_{(n,m) \in B_N} f(n,m) \, \mu(n,m) \geq \frac{f(N,0) \, \mu(N,0) }{(2N-1)^2} =   \frac{2^N}{(2N-1)^2},	
	\end{displaymath}
	which implies $Mf(1,0) = \infty$.
	
	On the other side, consider any ball $B$ containing $(-1,0)$ and denote
	\begin{displaymath}
	K = K(B) = \max\{n \in \mathbb{N} \colon (n,0) \in B\}.
	\end{displaymath}
	If $K \leq 0$, then $\sum_{(n,m) \in B} f(n,m) \, \mu(n,m) = 0$. In turn, if $K > 0$, then $B$ must contain at least one of the points $(0,-\lfloor K/2 \rfloor )$ and $(0,\lfloor K/2 \rfloor)$. Consequently, we have
	\begin{displaymath}
	\frac{1}{\mu(B)} \sum_{(n,m) \in B} f(n,m) \, \mu(n,m) \leq \frac{2 f(K,0)}{4^{\lfloor K/2 \rfloor}} \leq 4,
	\end{displaymath}
	which implies $Mf(-1,0) < \infty$. 
	
	\begin{example}
		Consider the space $(\mathbb{Z}^2, d_\infty, \mu)$, where $\mu$ is defined by
		\begin{displaymath}
		\mu(n,m) = \left\{ \begin{array}{rl}
		4^{|m|} & \textrm{if } n = 0,  \\
		
		2^{n^2} & \textrm{if } n < 0 \textrm{ and } m = 0,  \\
		
		1 & \textrm{otherwise. }  \end{array} \right. 
		\end{displaymath}
		Then both $M$ and $M^c$ do not possess the dichotomy property.
	\end{example}
	
	To verify that $M$ does not possess the dichotomy property we can use exactly the same function $f$ as in Example 3. It is easy to see that $Mf(1,0) = \infty$ and $Mf(-1,0) < \infty$ hold as before. Next, in order to show that $M^c$ does not possess the dichotomy property, let us take the function $g$ defined by
	
	\begin{displaymath}
	g(n,m) = \left\{ \begin{array}{rl}
	2^{n^2} & \textrm{if } n > 0 \textrm{ and } m = 0,  \\
	
	0 & \textrm{otherwise. }  \end{array} \right. 
	\end{displaymath}
	Consider the balls $B_N^+ = B_N(1,0)$ and $B_N^- = B_N(-1,0)$ for $N \in \mathbb{N}$. Observe that for large values of $N$ we have
	\begin{displaymath}
	\frac{1}{\mu(B_N^+)} \sum_{(n,m) \in B_N^+} g(n,m) \, \mu(n,m) \geq \frac{g(N,0)}{2 \mu(-N+2, 0)} = 2^{N^2 - (N-2)^2 - 1},
	\end{displaymath}
	and
	\begin{displaymath}
	\frac{1}{\mu(B_N^-)} \sum_{(n,m) \in B_N^-} g(n,m) \, \mu(n,m) \leq \frac{2 g(N-2,0)}{\mu(-N, 0)} = 2^{-N^2 + (N-2)^2 + 1}.
	\end{displaymath}
	Consequently, this easily leads to the conclusion that $Mg(1,0) = \infty$ and $Mg(-1,0) < \infty$. \newline
	
	At last, as we mentioned earlier, we will try to outline a sketch of how to adapt Examples 3 and 4 to the situation of $\mathbb{R}^2$ with the Euclidean metric. First of all, note that the key idea of Example 3 was to construct a measure which creates a kind of barrier separating (in the proper meaning) the points $(n,m)$ with positive and negative values of $n$, respectively. Exactly the same effect can be obtained if we define $w$ so that it behaves like $e^{|y|}$ in the strip $-\frac{1}{2} < |x| < \frac{1}{2}$ and like $1$ outside of it. However, because of some significant differences between the shapes of the balls determined by $d_e$ and $d_\infty$, respectively, one should be a bit more careful when looking for the proper function $f$ such that $Mf(x,y) = \infty$ if $x > 1$ and $Mf(x,y) < \infty$ if $x < -1$. Observe that any ball $B$ such that $(-1,0) \in B$ and $(N,0) \in B$ must contain at least one of the points $(0, - \sqrt{N})$ and $(0, \sqrt{N})$. Therefore, if $B_N$ is such that $N$ is the largest positive integer $n$ satisfying $(n,0) \in B_N$, then it would be advantageous to ensure that the integral $\int_{B_N} f(x,y) w(x,y) \, dx \, dy$ is no more than $C e^{\sqrt{N}}$, where $C > 0$ is some numerical constant. On the other hand, we want this quantity to tend to infinity with $N$ faster than $N^2$. This two conditions are fulfilled simultaneously if, for example, $f(x,y)$ behaves like $x^2$ in the region $\{(x,y) \in \mathbb{R}^2 \colon x > 0, -\frac{1}{2} < |y| < \frac{1}{2}\}$, and equals $0$ outside of it. 
	
	Finally, to arrange the situation of Example 4, it suffices to define $w$ in such a way that it is comparable to $e^{|y|}$ if $-\frac{1}{2} < |x| < \frac{1}{2}$, to $e^{x^2}$ if $x < 0$ and $-\frac{1}{2} < |y| < \frac{1}{2}$, and to $1$ elsewhere. Also, apart from those described above, there are no further difficulties in finding the appropriate functions $f$ and $g$ that break the dichotomy condition for $M$ and $M^c$, respectively.
	
\section{Necessary and sufficient condition}

The last section is mainly devoted to describing the exact characterization of situations, in which $M^c$ possesses the dichotomy property, for metric measure spaces of the form $(\mathbb{R}^d, d_e, \mu)$, $d \geq 1$, where $\mu$ is arbitrary. Namely, our goal is to prove the following.

\begin{theorem}
	Consider the metric measure space $(\mathbb{R}^d, d_e, \mu)$, $d \geq 1$, where $\mu$ is an arbitrary Borel measure. Then $M^c$ possesses the dichotomy property if and only if (\ref{C}) holds.
\end{theorem} 

We show the proof only for $d=2$, since in this case all the significant difficulties are well exposed and, at the same time, we omit a few additional technical details that arise when $d \geq 3$. In turn, the case $d=1$ is much simpler than the others, so we do not focus on it. When dealing with $\mathbb{R}^2$, we will write shortly $B_r(x,y)$ instead of $B_r((x,y))$, just like we did in the previous section in the context of $\mathbb{Z}^2$. 

\begin{proof}
	 First of all, let us recall that one of the implications has already been proven in Proposition 2. Thus, it is enough to show that (\ref{C}) is necessary for $M^c$ to possess the dichotomy property. 
	
	Take $(\mathbb{R}^2, d_e, \mu)$ and assume that (\ref{C}) fails to occur. Thus, for the point $(0,0)$ there exists a strictly increasing sequence of positive numbers $\{a_k\}_{k \in \mathbb{N}}$ such that
	\begin{displaymath}
	\mu(B_{a_k+1}(0,0)) \geq 2^{2k} \, \mu(B_{a_k}(0,0))
	\end{displaymath}
	holds for each $k \in \mathbb{N}$. In addition, we can force that $a_1 \geq 8$ and $a_{k+1} \geq a_k + 2$. For $n \in \mathbb{N}$ we introduce the auxiliary sets $S_{k+, j}^{(n)}$, $j \in \{1, \dots, 2^n\}$, defined by
	\begin{displaymath}
	S_{k+, j}^{(n)} = \Big \{(x,y) \in B_{a_k+1}(0,0) \colon \phi(x,y) \in \big[ \frac{2 \pi (j-1)}{2^n}, \frac{2 \pi j}{2^n} \big) \Big\},
	\end{displaymath}
	where $\phi(x,y) \in [0, 2\pi)$ is the angle that $(x,y)$ takes in polar coordinates.
	
	Take $n = 1$ and choose $j_1 \in \{1, 2\}$ such that the set
	\begin{displaymath}
	\Lambda_1 = \{k \in \mathbb{N} \colon \mu(S_{k+, j_1}^{(1)}) \geq \frac{1}{2} \mu(B_{a_{k}}(0,0)) \}
	\end{displaymath}
	is infinite. Next, take $n = 2$ and choose $j_2 \in \{1, 2, 3, 4\}$ satisfying $\lceil j_2/2 \rceil = j_1$ (where $\lceil \, \cdot \, \rceil$ is the ceiling function) and such that
	\begin{displaymath}
	\Lambda_2 = \{ k \in \Lambda_1 \colon \mu(B_{k+, j_2}^{(2)}) \geq \frac{1}{4} \mu(B_{a_{k}}(0,0)) \}
	\end{displaymath}
	is infinite. Continuing this process inductively we receive a sequence $\{j_n\}_{n \in \mathbb{N}}$ satisfying $\lceil j_{n+1} / 2 \rceil = j_n$, $n \in \mathbb{N}$, and, by invoking the diagonal argument, a strictly increasing subsequence $(a_{k_n})_{n \in \mathbb{N}}$ such that for each $n \in \mathbb{N}$ we have
	\begin{displaymath}
	\mu(S_{k_n+, j_n}^{(n)}) \geq \frac{1}{2^n} \mu(B_{a_{k_n}}(0,0)), \qquad n \in \mathbb{N}.
	\end{displaymath}
	
	From now on, for simplicity, we will write $B_n$ 
	and $S_{n+, j_n}$ instead of $B_{a_{k_n}}(0,0)$ 
	and $S_{k_n+, j_n}^{(n)}$, respectively. Observe that the received sequence $\{j_n\}_{n \in \mathbb{N}}$ determines a unique angle $\phi_0 \in [0, 2\pi)$ which indicates a ray around which, loosely speaking, a significant part of $\mu$ is concentrated. For the sake of clarity we assume that $\phi_0 = 0$ and therefore $\{j_n\}_{n \in \mathbb{N}}$ equals either $(1, 1, 1, \dots)$ or $(2, 4, 8, \dots)$.
	
	Denote $B_{n-} = B_{1/2}(-a_{k_n}+2, 0)$, $n \in \mathbb{N}$, and consider the function $f$ defined by
	\begin{displaymath}
	f = \sum_{n=1}^\infty \frac{2^n \mu(B_n)}{\mu(B_{n-})} \chi_{B_{n-}}.
	\end{displaymath} 
	Of course, $f \in L^1_{loc}(\mu)$. We will show that $M^cf(x_0,y_0) = \infty$ for $(x_0,y_0) \in B_{1/2}(0,0)$ and $M^cf(x_0,y_0) < \infty$ for $(x_0,y_0) \in B_{1/2}(3,0)$.
	
	Fix $(x_0,y_0) \in B_{1/2}(0,0)$ and observe that $B_{n-} \subset B_{a_{k_n}-1}(x_0,y_0) \subset B_n$ and therefore
	\begin{displaymath}
	\frac{1}{\mu(B_{a_{k_n}-1}(x_0,y_0))} \int_{B_{a_{k_n}-1}(x_0,y_0)} f \, d\mu \geq \frac{1}{\mu(B_n)} \int_{B_{n-}} f \, d\mu = 2^n,
	\end{displaymath}
	which implies $M^cf(x_0,y_0) = \infty$.
	
	In turn, fix $(x_0,y_0) \in B_{1/2}(3,0)$ and consider $r > 0$ such that $B_r(x_0,y_0)$ intersects at least one of the sets $B_{n-}$, $n \in \mathbb{N}$. Notice that this requirement forces $r>2$. We denote
	\begin{displaymath}
	N = N(r) = \max\{ n \in \mathbb{N} \colon B_r(x_0,y_0) \cap B_{n-} \neq \emptyset \}.
	\end{displaymath}
	One can easily see that this implies $r > a_{k_n}$ and hence $(a_{k_n},0) \in B_{r-2}(x_0,y_0)$. It is possible to choose $N_0 = N_0(x_0,y_0) \geq 2$ such that if $N \geq N_0$, then $(a_{k_N},0) \in B_{r-2}(x_0,y_0)$ implies $S_{N+,j_N} \subset B_r(x_0,y_0)$. Let $\tilde{N} = \max\{r>0 \colon N(r)<N_0\}$. If $2 < r \leq \tilde{N}$, then
	\begin{displaymath}
	\frac{1}{\mu(B_r(x_0,y_0))} \int_{B_r(x_0,y_0)} f \, d\mu \leq \frac{1}{\mu(B_2(x_0,y_0))} \int_{B_{\tilde{N}}(x_0,y_0)} f \, d\mu = C, 
	\end{displaymath}
	where $C$ is a numerical constant independent of $r$. On the other hand, if $r > \tilde{N}$, then
	\begin{displaymath}
	\frac{1}{\mu(B_r(x_0,y_0))} \int_{B_r(x_0,y_0)} f \, d\mu \leq \frac{2^{N+1} \mu(B_N)}{\mu(S_{N+,j_N})} \leq 2,
	\end{displaymath}
	which implies $M^cf(x_0,y_0) < \infty$.
\end{proof}

\begin{remark*}
	Note that this time the proof relies on some Euclidean geometry properties and therefore it cannot be repeated in a more general context. The only clearly visible way to generalize it is to replace the Euclidean metric. Indeed, one can, for example, put a metric $\rho$ induced by any norm on $\mathbb{R}^d$ in place of $d_e$ and get the desired result by following the same path only with a few minor modifications. Notice that in this case, of course, the balls in (\ref{C}) are taken with respect to $\rho$. Thus, among other things, we must take into account how the shape of these balls is related to the direction determined by the angle $\phi_0$ specified in the proof. Finally, the weak type $(1,1)$ inequality of $M^c$ associated to $(\mathbb{R}^d, \rho, \mu)$, which is needed to provide $\mu(\mathbb{R}^d \setminus L^c_f) = 0$ in Proposition 2, can be deduced from a stronger version of the Besicovitch Covering Lemma (see \cite[Theorem 2.8.14]{F}).
\end{remark*}

We conclude our studies with an example which indicates that a possible necessary and sufficient condition for $M$ must be of a completely different form. Namely, while condition (\ref{C}) concerned only the growth at infinity of a given measure, the parallel condition for non-centered operators should deal with both global and local aspects of the considered spaces. Thus, this problem, probably more difficult, is an interesting starting point for further investigation.

\begin{example}
Consider the space $(\mathbb{R}^2, d_e, \mu)$ with $\mu = \lambda_1 + \lambda_2$, where $\lambda_1$ is $1$-dimensional Lebesgue measure on $A = [0,1] \times \{0\}$ and $\lambda_2$ is $2$-dimensional Lebesgue measure on the whole plane. Then there exists $f \in L^1(\mu)$ with compact support such that $E_\infty(f) = A$. 
\end{example}

Indeed, denote $S_n = [0,1] \times (2^{-n^2}, 2^{-n^2+1})$ and consider the function
	\begin{displaymath}
	f = \sum_{n=1}^{\infty} 2^n \chi_{S_n}.
	\end{displaymath}
	Observe that $f$ equals $0$ outside the square $[0,1] \times [0,1]$ and $\|f\|_{1} = \sum_{n=1}^{\infty} 2^n \cdot 2^{-n^2} \leq 2$. 
	
	Let us fix $x_0 \in [0,1]$ and consider the balls $B_n = B_{2^{-n^2 + \epsilon_n}}(x_0, 2^{-n^2})$, $n \in \mathbb{N}$, where $\epsilon_n > 0$ are such that $\mu(B_n) \leq 2^{-2n^2+2}$. Observe that $(x_0,0) \in B_n$ for each $n$. If $n \geq 2$, then $\mu(B_n \cap S_n) \geq 2^{-2n^2 - 1}$ and, consequently,
	\begin{displaymath}
	\frac{1}{\mu(B_n)} \int_{B_n} f \, d\mu \geq \frac{2^n \cdot 2^{-2n^2 - 1}}{2^{-2n^2+2}} = 2^{n-3}, 
	\end{displaymath}
	which implies $Mf(x_0,0) = \infty$. 
	
	On the other hand, consider $(x_0, y_0) \notin A$. In this case, there exist $\epsilon > 0$ and $L > 0$ such that $d_e((x_0,y_0),(x,y)) < \epsilon$ implies $f(x,y) \leq L$ and, as a result, we obtain $Mf(x_0,y_0) \leq \max\{L, 2/\lambda_2(B_{\epsilon/2}(x_0,y_0))\} < \infty$.

\section*{Acknowledgement}
This article was largely inspired by the suggestions of Professor Krzysztof Stempak. I would like to thank him for insightful comments and continuous help during the preparation of the paper.

\end{document}